\documentclass[smallextended]{svjour3}       
\usepackage[utf8]{inputenc}

\usepackage{color,verbatim}
\usepackage{amsmath}
\usepackage{amssymb}
\usepackage{times}

\usepackage{graphicx}

\usepackage{epstopdf}
\usepackage[caption=false]{subfig}
\usepackage[numbers,sort&compress]{natbib}
\makeatletter
\def\NAT@def@citea{\def\@citea{\NAT@separator}}%

\newtheorem{assumption}[theorem]{Assumption}
\newtheorem{corrolary}[theorem]{Corollary}

\newcommand{\rset}{\mathbb{R}}

\usecounter{algorithmenumi}

\providecommand{\norm}[1]{\lVert#1\rVert}

\newcommand{\be}{\begin{equation}}
\newcommand{\ee}{\end{equation}}
\newcommand{\eref}[1]{(\ref{#1})}

\newcommand\pdffig[4][7cm]{
	\begin{figure}[t]
		\centering
		\includegraphics[width=#1]{#2}
		\caption{#3}
		\label{#4}
	\end{figure}
}

\begin{document}

\title{Stochastic proximal splitting algorithm for composite minimization}


\author{Andrei Patrascu \and Paul Irofti} 
	\institute{A. Patrascu and P. Irofti are with the Research Center for Logic, Optimization and Security (LOS), Department of Computer Science, Faculty of Mathematics and Computer Science, University of Bucharest. (e-mails: andrei.patrascu@fmi.unibuc.ro, paul@irofti.net).
	The research of A. Patrascu was supported by a grant of the Romanian Ministry of Education and Research, CNCS - UEFISCDI, project number PN-III-P1-1.1-PD-2019-1123, within PNCDI III. Also, the research work of P. Irofti was supported by a grant of the Romanian Ministry of
	Education and Research, CNCS - UEFISCDI, project number
	PN-III-P1-1.1-PD-2019-0825, within PNCDI III.}

\date{Received: date / Accepted: date}

\maketitle

\begin{abstract}
Supported by the recent contributions in multiple branches, the first-order splitting algorithms became central for structured nonsmooth optimization. In the large-scale or noisy contexts, when only stochastic information on the objective function is available, the extension of proximal gradient schemes to stochastic oracles is heavily based on the tractability of the proximal operator corresponding to nonsmooth component, which has been deeply analyzed in the literature. However, there remained some questions about the difficulty of the composite models where the nonsmooth term is not proximally tractable anymore. Therefore, in this paper we tackle composite optimization problems, where the access only to stochastic information on both smooth and nonsmooth components is assumed, using a stochastic proximal first-order scheme with stochastic proximal updates. We provide sublinear $\mathcal{O}\left( \frac{1}{k} \right)$ convergence rates (in expectation of squared distance to the optimal set) under the strong convexity assumption on the objective function. Also, linear convergence is achieved for convex feasibility problems. The empirical behavior is illustrated by numerical tests on parametric sparse representation models.
\end{abstract}

\section{Introduction}

In this paper we consider the following convex composite optimization problem:
\begin{align*}
\min\limits_{x \in \rset^n} & \;\;  f(x) + h(x),
\end{align*}
where $f$ is the smooth component and $h$ is a proper, convex, lower-semicontinuous function.
In literature, many applications from statistics \cite{RosVil:14} or signal processing often motivate noisy contexts allowing access only to stochastic first order information on smooth function $f$, having regularizer $h$ as a typical proximally-tractable convex function. By proximally-tractable we mean that the proximal map of a given function is computable in closed form or, at most, in linear time. Therefore, in these situations the following stochastic model 
\begin{align*}
\min\limits_{x \in \rset^n} & \;\; \mathbb{E}[f(x;\xi)] + h(x),
\end{align*}
where $\xi$ is a random variable,
express better the previous real assumptions.
However, the recent dimensionality inflation of machine learning \cite{HaLes:15,ZhoKwo:14,ShiLin:15} and signal processing \cite{StoIro:19, ShiLin:15} models gave birth to optimization problems with complicated regularizers or complicated many constraints. As practical examples we recall: parametric sparse represention \cite{StoIro:19}, group lasso \cite{ZhoKwo:14, ShiLin:15,HaLes:15}, CUR-like factorization \cite{WanWan:17}, graph trend filtering \cite{SalBia:17,VarLee:19}, dictionary learning \cite{YanEla:16, StoIro:19}. Motivated by all these models, in this paper instead of assuming proximal-tractable regularizer $h$, we consider that $h$ is expressed as an expectation of stochastic proximally-tractable components $h(\cdot;\xi)$ (i.e. $h(x) = \mathbb{E}[h(x;\xi)]$). Thus we focus on the following model:
\begin{align}
\label{problem_intro}
\min\limits_{x \in \rset^n} & \;\;  F(x) := \mathbb{E}_{\xi \in \Omega} [f(x;\xi)] + \mathbb{E}_{\xi \in \Omega} [h(x;\xi)],
\end{align}
where $\xi$ is a random variable associated with probability space $(\mathbb{P},\Omega)$. Functions $f(\cdot;\xi): \rset^n \to \rset^{}$ are smooth with Lipschitz gradients and $h(\cdot, \xi): \rset^n \mapsto (-\infty, + \infty]$ are proper convex and lower-semicontinuous, $\mathbb{E}[\cdot]$ is the expectation over respective random variable. 
In general, many existing primal schemes encounter computational difficulties when a large (possibly infinite) number  of constraints are present, since they are based on full projections onto complicated feasible set.

\vspace{5pt}

\noindent \textbf{Contributions}. $(i)$ We analyze a stochastic first-order splitting scheme relying on stochastic gradients and stochastic proximal updates, which naturally generalize the widely known Stochastic Gradient Descent (SGD) and Stochastic Proximal Point (SPP) algorithms toward composite models with untractable regularizations; $(ii)$ we provide $\mathcal{O}(\frac{1}{k})$ iteration complexity estimates which were previously unknown for this type of schemes. In particular, for convex feasibility problems, the analysis yields naturally linear convergence rates.

\vspace{5pt}

\noindent We briefly recall further the milestone results from stochastic optimization literature with focus on the complexity of stochastic first-order methods.

\subsection{Previous work}

Great attention has been given in the last decade to the behaviour of stochastic first order schemes, with special focus in stochastic gradient descent (SGD), on a variety of models under different convexity properties, see \cite{NemJud:09,MouBac:11,ShaSin:11,NguNgu:18,RosVil:14,Ned:10,Ned:11}. Since the analysis of SGD naturally require a typical smoothness, appropriate extensions are necessary to attack nonsmooth models. These extensions are embodied by the stochastic proximal point (SPP) algorithm, which has been recently analyzed using various differentiability assumptions, see \cite{TouTra:16,RyuBoy:16,Bia:16,PatNec:18,KosNed:13,WanBer:16,AsiDuc:18} and has shown surprinsing analytical and empirical performances.
In \cite{TouTra:16} is considered the typical  stochastic learning model involving the expectation of random particular components $f(\cdot;\xi)$ defined by the composition of a smooth function and a linear operator, i.e.: $  f(x;\xi) = \ell(a^T_{\xi} x), $
where $a_\xi \in \rset^n$.  Their complexity analysis requires smoothness and strong convexity to obtain in the quadratic mean and
an $\mathcal{O}\left(\frac{1}{k^{\gamma}} \right)$  convergence rate, using vanishing stepsize. The generalization of these convergence guarantees is undertaken in \cite{PatNec:18}, where no linear composition structure is required and an (in)finite number of constraints are included in the stochastic model, i.e. $h(\cdot;\xi) = \mathbb{I}_{X_{\xi}}(\cdot)$. However, the analysis of \cite{PatNec:18} requires strong convexity and Lipschitz gradient continuity for each functional component $f(\cdot;\xi)$. Note that our analysis surpasses this restriction and provides a natural generalization of \cite{PatNec:18} to nonsmooth composite models. Further, in  \cite{Bia:16} a  general asymptotic convergence analysis of slightly modified SPP scheme has been provided, under mild convexity assumptions on a finitely constrained stochastic problem.  In \cite{WanBer:16} similar SPP algorithms are developed, which are also tailored for complicated feasible sets. The authors focus on mild convex models deriving optimal $\mathcal{O}(1/\sqrt{k})$ rates. 
Recently, in \cite{AsiDuc:18}, the authors analyze SPP schemes for stochastic models with "shared minimizers" obtaining linear convergence results, for variable stepsize SPP. Also, without shared minimizers assumption, they obtain for SPP $\mathcal{O}\left( \frac{1}{\sqrt{k}} \right)$ in convex Lipschitz continuous case and, furthermore, $\mathcal{O}\left( \frac{1}{k} \right)$ in strongly convex case. 

\vspace{5pt}

\noindent Splitting first-order schemes received significantly attention due to their natural insight and simplicity in contexts when a sum of two components are minimized (see \cite{Nes:13,BecTeb:09}). However, when the information access is limited only to stochastic samples of the two components, extending the existing guarantees is not straightforward. 
Notice that overall, on one hand, SPP avoids any splitting in composite models by treating the constrained smooth problems as black box expectations (see \cite{TouTra:16,PatNec:18,AsiDuc:18}). On the other hand, until recently the composite nonsmooth models assumed proximal-tractable $h$ in order to extend proof arguments of the results related to stochastic smooth optimization \cite{RosVil:14}. Therefore, only recently the full stochastic composite models with stochastic regularizers have been properly tackled \cite{SalBia:17}, where almost sure asymptotic convergence is established for a stochastic splitting scheme, where each iteration represents a typical proximal gradient update with respect to stochastic samples of $f$ and $h$. In our paper we analyze the nonasymptotic behaviour of this scheme and point the relations with other algorithms from the literature.

The stochastic splitting schemes are also related to the model-based methods developed in \cite{DavDru:19}. Here, the authors developed a unified algorithmic framework, which generates stochastic algorithms, for different models arising in learning applications. They assume their composite objective function to be the sum of a (weakly convex) stochastic component with bounded gradients and simple (proximally tractable) convex regularization. Although their framework is algorithmically more general, our analysis avoid these boundedness assumptions, allows objectives with a component having Lipschitz continuous gradient and do not require proximal tractability on regularizations.



\vspace{5pt}

\noindent \textbf{Notations}. We use notation $[m] =\{1,\cdots, m\}$. For $x,y \in \rset^n$ denote the scalar product  $\langle x,y \rangle = x^T y$ and Euclidean norm by $\|x\|=\sqrt{x^T x}$. The projection operator onto set $X$ is denoted by $\pi_X$ and the distance from $x$ to the set $X$ is denoted $\text{dist}_X(x) = \min_{z \in X} \norm{x-z}$. 
The indicator function of a set $X$ is denoted: $\mathbb{I}_{X}(x) = \begin{cases}
0, & \text{if} \; x \in X \\
\infty , & \; \text{otherwise}
\end{cases}$.
We use notations $\partial h(x;\xi)$ for the subdifferential set and $g_h(x;\xi)$ for a subgradient  of $h(\cdot;\xi)$ at $x$. In differentiable case, $\nabla f(\cdot;\xi)$ is the gradient of component $\xi$. Finally, we use $\mathbb{E}[\cdot]$ for (conditional) expectation operator.

\subsection{Preliminaries}
We denote the set of optimal solutions with $X^*$ and $x^*$ any optimal point for \eqref{problem_intro}. 
\begin{assumption} \label{assump_basic}
The objective function of our main problem \eqref{problem_intro} satisfies:

\noindent $(i)$ The function $f(\cdot;\xi)$ has $L_f$-Lipschitz gradient, i.e. there exists $L_f > 0$ such that:
\begin{align*}
\norm{\nabla f(x;\xi) - \nabla f(y;\xi)} \le L_f \norm{x-y}, \qquad \forall  x,y \in \rset^n, \xi \in \Omega.
\end{align*}
and $f$ is $\sigma_f-$strongly convex,	i.e. there exists $\sigma_f > 0$ satisfying:
\begin{align}\label{strong_convexity_on_f}
f(x) \ge f(y) + \langle \nabla f(y), x-y\rangle + \frac{\sigma_f}{2}\norm{x-y}^2 \qquad \forall x,y \in \rset^n.
\end{align}



\noindent $(ii)$ There exists subgradient mappings $g_F(\cdot;\xi): \text{dom} \; F(\cdot;\xi) \times \Omega \mapsto \rset^n$ such that $g_F(x;\xi) \in \partial F(x;\xi), \forall \xi \in \Omega$ and $\mathbb{E}[g_F(x;\xi)] \in \partial F(x).$ 

\noindent $(iii)$  $F(\cdot;\xi)$ has bounded gradients on the optimal set: there exists $\mathcal{S}^*_F \ge 0$ such that $\mathbb{E}\left[\norm{g_F(x^*;\xi)}^2\right]  \le \mathcal{S}^*_F  < \infty$ for all $x^* \in X^*$.

\noindent $(iv)$ For any $g_F(x^*) \in \partial F(x^*)$ there exists bounded subgradients $g_F(x^*;\xi) \in \partial F(x^*;\xi),$ such that $ \mathbb{E}[g_F(x^*;\xi)] = g_F(x^*)$ and $\mathbb{E}[\norm{g_F(x^*;\xi)}^2]<\mathcal{S}^*_F$. Moreover, for simplicity, we assume throughout the paper $g_F(x^*) = \mathbb{E}[g_F(x^*;\xi)] = 0$.
\end{assumption}

\noindent The first part of the above assumption is natural in the stochastic optimization problems. The Assumption $(ii)$ guarantee the existence of a subgradient mapping. The third part Assumption $(iii)$ is standard in the literature related to proximal stochastic algorithms. Notice that most stochastic subgradient algorithms require bounded gradients on the entire domain \cite{MouBac:11,NemJud:09,NguNgu:18}, which is more restrictive than condition $(iii)$.  The assumption $(iv)$ needs a more consistent discussion, detailed in \cite{Pat:20}. However, for completeness we include it in the remark below.

\begin{remark}
Denote the set $\mathbb{E}[\partial F(\cdot;\xi)] = \left\{\mathbb{E}[g_F(\cdot;\xi)] \;|\; g_F(\cdot;\xi) \in \partial F(\cdot;\xi)  \right\}$. In general, for convex functions it can be easily shown $\mathbb{E}\left[ \partial F(x;\xi)\right] \subseteq \partial F(x)$ for all $x \in \text{dom}(F)$ (see \cite{RocWet:82}). However, $(iv)$ is guaranteed by the stronger equality: \begin{align}\label{diff_equality}
\mathbb{E}\left[ \partial h(x;\xi)\right] = \partial h(x).
\end{align}

\noindent \textbf{Discrete case.} Let us consider finite discrete domains $\Omega = \{1, \cdots, m\}$. Then \cite[Theorem 23.8]{Roc:70} guarantees that the finite sum  objective function from \eqref{problem_intro} satisfy \eqref{diff_equality} if $ \bigcap\limits_{\xi \in \Omega}\text{ri}(\text{dom}(h(\cdot;\xi))) \neq \emptyset$.
The $\text{ri}(\text{dom}(\cdot))$ can be further relaxed to $\text{dom}(\cdot)$ for polyhedral components. 
In particular,  let $X_1, \cdots, X_m$ be finitely many closed convex satisfying qualification condition: $\bigcap\limits_{i=1}^m \text{ri}(X_i) \neq \emptyset$, then also \eqref{diff_equality} holds, i.e. $\mathcal{N}_{X}(x) = \sum\limits_{i=1}^m \mathcal{N}_{X_i}(x)$ (see ( by \cite[Corrolary 23.8.1]{Roc:70})). Again, $\text{ri}(X_i)$ can be relaxed to the set itself for polyhedral sets. 
As pointed by \cite{BauBor:99}, the (bounded) linear regularity property of $\{X_i\}_{i=1}^m$ implies the intersection qualification condition. 

Under support of these arguments, observe that $(iv)$ can be easily checked for most finite-sum examples arisen in convex learning problems.

\noindent \textbf{Continuous case.}
In the nondiscrete case, sufficient conditions for \eqref{diff_equality} are discussed in \cite{RocWet:82}. Based on the arguments from \cite{RocWet:82}, an assumption equivalent to $(iv)$ is considered in \cite{SalBia:17} under the name of $2-$integrable representation of $x^*$ (definition in \cite[Section B]{SalBia:17}). On short, if $h(\cdot;\xi)$ is normal convex integrand with full domain then $x^*$  admits an $2-$integrable representation $\{g_h(x^*;\xi)\}_{\xi \in \Omega}$, and implicitly $(iv)$ holds. 

We mention that deriving a more complicated result similar to Lemma \ref{lemma_auxres} we could avoid assumption $(iv)$. However, since $(iv)$ facilitates the simplicity and naturality of our results and while our target applications are not excluded, we assume throughout the paper that $(iv)$ holds. 
\end{remark}


\noindent Let closed convex sets $\{X_{\xi}\}_{\xi \in \Omega}$ and $X = \bigcap_{\xi \in \Omega} X_{\xi}$, then a favorable "conditioning" property for most projection methods is the linear regularity property: there exists $\kappa > 0$ such that
\begin{align}\label{linear_regularity}
\mathbb{E}[\text{dist}_{X_{\xi}}^2(x)] \ge \kappa \text{dist}_{X}^2(x) \qquad \forall x \in \rset^n.
\end{align}


\noindent Given some smoothing parameter $\mu>0$, define Moreau envelope of $h(x;\xi)$ and the prox operator as follows:
\begin{align*}
 h_{\mu} (x;\xi) & := \min\limits_{z \in \rset^n}  \;  h(z;\xi) + \frac{1}{2\mu} \norm{z - x}^2 \\
 \text{prox}_{h,\mu}(x;\xi) & := \arg\min\limits_{z \in \rset^n} h(z;\xi) + \frac{1}{2\mu} \norm{z - x}^2.
\end{align*} 
The approximate $h_{\mu}(\cdot;\xi)$ inherits the same convexity properties of $h(\cdot;\xi)$ and additionally has Lipschitz continuous gradient with constant $\frac{1}{\mu}$, see \cite{RocWet:98}. In particular, when $h(x;\xi) = \mathbb{I}_{X_{\xi}}(x)$ the prox operator becomes the projection operator $\text{prox}_{h,\mu}(x;\xi) = \pi_{X_{\xi}}(x)$.

\section{Stochastic Splitting Proximal Gradient Algorithm}

\noindent In the following section we present the Stochastic Splitting Proximal Gradient (SSPG) algorithm and analyze its nonasymptotic convergence towards the optimal set of the original problem \eqref{problem_intro}. The asymptotic convergence of vanishing stepsize SSPG have been analyzed in \cite{SalBia:17}. 

\vspace{5pt}

\noindent Let $x^0 \in \rset^n$ be a starting point and $\{\mu_k\}_{k \ge 0}$ be a nonincreasing positive sequence of
stepsizes.

\begin{flushleft}
\textbf{ Stochastic Splitting Proximal Gradient algorithm (SSPG)}: \quad For $k\geq 0$ compute \\
1. Choose randomly $\xi_k \in \Omega$ w.r.t. probability distribution $\mathbb{P}$\\
2. Update: 
\begin{align*}
y^{k}  & =  x^k - \mu_k \nabla f(x^k;\xi_k) \\
x^{k+1}  & = \text{prox}_{h,\mu_k}(y^{k};\xi_k)
\end{align*}
3. If the stoppping criterion holds, then \textbf{STOP}, otherwise $k = k+1$.
\end{flushleft}


\noindent The SSPG iteration $x^{k+1} = \text{prox}_{h,\mu_k} \left(x^k - \mu_k \nabla f(x^k;\xi_k);\xi_k \right)$ is mainly a Stochastic Proximal Gradient iteration based on stochastic proximal maps \cite{SalBia:17}. Thus, the first step of algorithm SSPG consists of a varying-stepsize (vanilla) stochastic gradient update, while the second step rely on a stochastic proximal update, or equivalently a gradient step in the direction of the randomly sampled gradient of expected Moreau envelope $h_{\mu}(\cdot)$. Further results will state that a diminishing stepsize is an appropriate choice to obtain convergence in expectation.
By varying our central model, this general SSPG scheme recovers several well-known stochastic first order algorithms. 

\noindent $(i)$ In the smooth case ($h = 0$), SSPG reduces to vanilla SGD \cite{MouBac:11}:
$$x^{k+1}  =  x^k - \mu_k \nabla f(x^k;\xi_k).$$
 
\noindent $(ii)$ By considering proximal-tractable regularizers (i.e. $h(\cdot;\xi) = h(\cdot)$) or simple convex sets (i.e. $h(\cdot;\xi) = \mathbb{I}_{X}(\cdot)$, with $\pi_X(\cdot)$ computable in closed form),  then we recover Proximal (or Projected, respectively) SGD \cite{RosVil:14}: 
$$x^{k+1}  =  \text{prox}_{h,\mu_k}\left(x^k - \mu_k \nabla f(x^k;\xi_k)\right) \quad \text{or} \quad x^{k+1}  =  \pi_{X}\left(x^k - \mu_k \nabla f(x^k;\xi_k)\right).$$ 

\noindent $(iii)$ For nonsmooth objective functions, when $f = 0$, SSPG is equivalent with SPP iteration \cite{PatNec:18, AsiDuc:18,TouTra:16}: 
$$x^{k+1}  =  \text{prox}_{h,\mu_k}(x^{k};\xi_k).$$ 

\noindent $(iv)$  For CFPs, i.e. $f = 0, h(\cdot;\xi) = \mathbb{I}_{X_{\xi}}(\cdot)$, the SSPG iteration is simplified since $y^k = x^k$ and  $x^{k+1} = \text{prox}_{h,\mu_k}(x^k;\xi_k)$. Thus it recovers Randomized Alternating Projections scheme \cite{BauDeu:03}:
$$x^{k+1}  =  \pi_{X_{\xi_k}}(x^{k}).$$

\noindent Therefore, the below results will implicitly represent unifying convergence rates for these algorithms, under stated assumptions.


\section{Iteration complexity in expectation}

\noindent In this section we assume that the function $f$ is strongly convex and derive convergence rates for this particular case.
We recall a few elementary inequalities that will be used in our proofs. The proof of the following lemma can be found in \cite{Nes:04}.
\begin{lemma}[\cite{Nes:04}]\label{lemma_Lipschitz}
Under Assumption \ref{assump_basic}$(i)$, the following inequality holds:
\begin{align*}
 f(x;\xi) \le f(y;\xi) + \langle \nabla f(y;\xi), x-y\rangle + \frac{L_f}{2}\norm{x-y}^2 \qquad \forall x,y \in \rset^n.
\end{align*}
\end{lemma}
A second elementary bound the we will use is: let $v,x,y \in \rset^n$ and $\nu > 0$, then
\begin{align}\label{elemineq_lowbound}
\langle v, x-y\rangle + \frac{1}{2\nu}\norm{x-y}^2 \ge -\frac{\nu}{2}\norm{v}^2.
\end{align}
\noindent We denote the history of index choices by $\Xi_k = \{\xi_0, \cdots, \xi_{k-1}\}$.

\begin{lemma}\label{lemma_recurence}
Let Assumption \ref{assump_basic} hold and $\mu_k \le \frac{1}{2L_f}$. Then the sequence $\{x^k\}_{k \ge 0}$ generated by SSPG satisfies:
\begin{align*}
\mathbb{E}[\; \norm{x^{k+1} - x^*}^2] \le  & \; (1 - \sigma_f \mu_k) \;\mathbb{E}[\; \norm{x^k-x^*}^2] \\ 
& +  2\mu_k \mathbb{E}\left[ F(x^*) - F(x^{k+1};\xi_k) - \frac{1}{4\mu_k} \norm{x^{k+1} - x^k}^2\right].
\end{align*}
\end{lemma}
\begin{proof}
First notice that from the optimality conditions of the subproblem $\min_z h(z;\xi) + \frac{1}{2\mu}\norm{z - y}^2$, the following relation holds:
\begin{align}\label{optcond_h}
    g_h(x^{k+1};\xi_k) + \frac{1}{\mu_k}\left(x^{k+1} - y^{k}\right) = 0.
\end{align}	
Further we obtain the main recurrence:
\begin{align}
&\norm{x^{k+1} - x^*}^2 
 = \norm{x^k-x^*}^2  +
2 \langle x^{k+1} - x^k, x^k-x^* \rangle + \norm{x^{k+1} - x^k}^2 \nonumber\\
& = \norm{x^k-x^*}^2 +
2 \langle x^{k+1} - x^k, x^{k+1}-x^* \rangle - \norm{x^{k+1} - x^k}^2 \nonumber\\
& = \norm{x^k-x^*}^2 +
2 \langle \mu_k \nabla f(x^k;\xi_k) + \mu_k g_h(x^{k+1};\xi_k), x^* - x^{k+1} \rangle \!-\! \norm{x^{k+1} - x^k}^2 \nonumber\\
& \le \norm{x^k-x^*}^2 +
2\mu_k  \langle \nabla f(x^k;\xi_k), x^* - x^{k+1} \rangle - \norm{x^{k+1} - x^k}^2 \nonumber\\
& \hspace{5cm} + 2\mu_k \;[\; h(x^*;\xi_k) - h(x^{k+1};\xi_k) \; ] \nonumber\\
& = \norm{x^k-x^*}^2 -
2\mu_k \bigg( \langle \nabla f(x^k;\xi_k), x^{k+1} - x^k \rangle + \frac{1}{4\mu_k}\norm{x^{k+1} - x^k}^2 \label{prelimrecc_rel1} \\
&  + h(x^{k+1};\xi_k)\bigg)  + 2\mu_k \langle \nabla f(x^k;\xi_k), x^* - x^{k}\rangle - \frac{1}{2}\norm{x^{k+1} - x^k}^2 + 2\mu_k h(x^*;\xi_k).\nonumber
\end{align}
By using in \eqref{prelimrecc_rel1} the stepsize bound $\mu_k \le \frac{1}{2L_f}$ and Lemma \ref{lemma_Lipschitz}, we obtain:
\begin{align*}
&\norm{x^{k+1} - x^*}^2 \overset{\mu_k \le \frac{1}{2L_f}}{\le} \norm{x^k-x^*}^2\\
& \hspace{0.5cm}  -2\mu_k \bigg( \langle \nabla f(x^k;\xi_k), x^{k+1} - x^k \rangle + \frac{L_f}{2}\norm{x^{k+1} - x^k}^2 + h(x^{k+1};\xi_k)\bigg) \\
& \hspace{0.5cm} + 2\mu_k \langle \nabla f(x^k;\xi_k), x^* - x^{k}\rangle - \frac{1}{2}\norm{x^{k+1} - x^k}^2 + 2\mu_k h(x^*;\xi_k)\\
& \overset{Lemma \; \ref{lemma_Lipschitz}}{\le} \norm{x^k-x^*}^2 -2\mu_k \bigg(  F(x^{k+1};\xi_k) - f(x^k;\xi_k)\bigg) \\
& \hspace{0.5cm} + 2\mu_k \langle \nabla f(x^k;\xi_k), x^* - x^{k}\rangle - \frac{1}{2}\norm{x^{k+1} - x^k}^2 + 2\mu_k h(x^*;\xi_k).
\end{align*}
Further we take expectation w.r.t. $\xi_k$ in both sides and use the strong convexity property \eqref{strong_convexity_on_f} (with $x = x^k$ and $y = x^*$) to finally derive the following:
\begin{align*}
& \mathbb{E}[\norm{x^{k+1} - x^*}^2 \;|\; \Xi_k] 
 \le  \norm{x^k-x^*}^2 
+ 2\mu_k \mathbb{E}[\left( f(x^k;\xi_k) - F(x^{k+1};\xi_k) \right) \;|\; \Xi_k] \nonumber\\
&\hspace{2cm} + 2\mu_k \langle \nabla f(x^k), x^* - x^{k}\rangle - \frac{1}{2}\mathbb{E}\left[\norm{x^{k+1} - x^k}^2 \;|\; \Xi_k \right] + 2\mu_k h(x^*) \nonumber \\
& \le \norm{x^k-x^*}^2 +
2\mu_k \mathbb{E}\left[ f(x^k;\xi_k) - F(x^{k+1};\xi_k)  \;|\; \Xi_k \right] \nonumber \\
& \hspace{1.5cm} + 2\mu_k \left( F(x^*) - f(x^k) \!-\! \frac{\sigma_f}{2}\norm{x^k-x^*}^2 \right) \!-\! \frac{1}{2}\mathbb{E}\left[\norm{x^{k+1} - x^k}^2 \;|\; \Xi_k \right] \nonumber \\
& \!=\! (1 \!-\! \sigma_f \mu_k)\norm{x^k-x^*}^2 \!+\!  2\mu_k \mathbb{E}\left[ F(x^*) \!-\! F(x^{k+1};\xi_k) \!-\! \frac{1}{4\mu_k} \norm{x^{k+1} \!-\! x^k}^2 \;|\; \Xi_k \right]\!\!.
\end{align*}
Finally, by taking full expectation $\mathbb{E}[\; \cdot \;]$, over the entire index history, we obtain the above result.
\end{proof} 


\noindent Further we present some lower bounds on the second term from the right hand side, which will allow the synthesis of general complexity estimates over stochastic and deterministic contexts.

\begin{lemma}\label{lemma_auxres}
	Given $\mu > 0$, let Assumption \ref{assump_basic} hold. Then $F$ satisfies the following relations: given $k \ge 0$
	\begin{enumerate}
		\item[$(i)$] $ \mathbb{E}\left[F(x^{k+1};\xi_k) - F^* \!+\! \frac{1}{4\mu_k}\norm{x^{k+1} - x^k}^2 \right] \ge - \mu_k \mathbb{E} \left[  \norm{g_F(x^*;\xi)}^2\right].$  	
		\item[$(ii)$] Let $\{X_{\xi}\}_{\xi \in \Omega}$ be convex sets satisfying linear regularity with constant $\kappa$, such that $X := \bigcap_{\xi \in \Omega} X_{\xi} \neq \emptyset$. Also let $h(\cdot;\xi) = \mathbb{I}_{X_{\xi}}(\cdot)$, then  
		\begin{align*}
		& \mathbb{E}\left[F(x^{k+1};\xi_k) - F^* + \frac{1}{4\mu_k}\norm{x^{k+1} - x^k}^2\right] \\ 
		& \hspace{1cm} \ge - 2\mu_k\mathbb{E}\left[\norm{\nabla f(x^*;\xi)}^2 \right] -  \frac{4\mu_k}{\kappa}\norm{\nabla f(x^*)}^2 + \frac{\kappa}{16\mu_k}\mathbb{E}\left[\text{dist}_X^2(x^k)\right]
		\end{align*}
	\end{enumerate}
\end{lemma}
\begin{proof}
In order to prove $(i)$ let $x^* \in X^*$ and $g_F(x^*;\xi) \in \partial F(x^*;\xi)$, by convexity of $f(\cdot;\xi)$ we have:
	\begin{align}
	 & F(x^{k+1};\xi_k) - F^* + \frac{1}{4\mu_k}\norm{x^{k+1} - x^k}^2 \nonumber\\
	& \ge \langle g_F(x^*;\xi_k), x^{k+1} - x^*  \rangle + \frac{1}{4\mu_k} \norm{x^{k+1} - x^k}^2  \nonumber\\
	& \ge \langle g_F(x^*;\xi_k), x^k - x^* \rangle + \langle g_F(x^*;\xi_k), x^{k+1} - x^k \rangle + \frac{1}{4\mu_k} \norm{x^{k+1} - x^k}^2.  \label{rel_for_(i)}
	\end{align}
	On one hand, based on Assumption \ref{assump_basic} $(iv)$, we observe that the first term of the right hand side has null expectation:
	\begin{align}\label{null_term_rhs}
	\mathbb{E} \left[\langle g_F(x^*;\xi_k),x^k-x^*\rangle \right] 
	& = \mathbb{E} \left[ \; \left\langle \mathbb{E} \left[g_F(x^*;\xi_k) \;|\; \Xi_k \right],x^k-x^* \right\rangle \right] \nonumber\\
	& \overset{\text{Assump.}\; \ref{assump_basic}(iv)}{=} \mathbb{E} \left[ \; \left\langle  g_F(x^*) ,x^k-x^* \right\rangle \right] = 0. 
	\end{align}
	On the other hand, for any $x^* \in X^*$, the second term can be lower bounded by \eqref{elemineq_lowbound}:
	\begin{align}\label{gradterm_rhs}
	& \langle g_F(x^*;\xi_k), x^{k+1} - x^k \rangle + \frac{1}{4\mu_k} \norm{x^{k+1} - x^k}^2 \ge - \mu_k  \norm{g_F(x^*;\xi_k)}^2.
	\end{align}
	We take full expectation (over the entire index history) in \eqref{rel_for_(i)} and use relations \eqref{null_term_rhs}-\eqref{gradterm_rhs} to get the first part $(i)$. 
	
	\vspace{10pt}
	
	\noindent To derive the second part $(ii)$, we proceed slightly different:
	\begin{align}
	& F(x^{k+1};\xi_k) - F^* + \frac{1}{4\mu_k}\norm{x^{k+1} - x^k}^2 \nonumber\\
	& = f(x^{k+1};\xi_k) - f^* + \frac{1}{8\mu_k}\norm{x^{k+1} - x^k}^2 +  h(x^{k+1};\xi_k) + \frac{1}{8\mu_k}\norm{x^{k+1} - x^k}^2 \nonumber\\
	& \ge f(x^{k+1};\xi_k) - f^* + \frac{1}{8\mu_k}\norm{x^{k+1} - x^k}^2 +  \min_z h(z;\xi_k) + \frac{1}{8\mu_k}\norm{z - x^k}^2 \nonumber\\
	&\ge \langle \nabla f(x^*;\xi_k), x^{k+1} - x^* \rangle + \frac{1}{8\mu_k}  \norm{x^{k+1} - x^k}^2 + h_{4\mu_k}(x^k;\xi_k)   \nonumber \\
	& = \langle \nabla f(x^*;\xi_k), x^k - x^* \rangle + \nonumber \\
	& \hspace{1cm} \langle \nabla f(x^*;\xi_k), x^{k+1} - x^k \rangle + \frac{1}{8\mu_k}  \norm{x^{k+1} - x^k}^2 + h_{4\mu_k}(x^k;\xi_k). \label{rel_for_(ii)}
	\end{align}
	We proceed to bound each term of right hand side, similarly as in $(i)$. By using the optimality condition (i.e. $\langle \nabla f(x^*), z - x^*\rangle \ge 0 $ for all $z \in X$), the expectation of the first term can be lower bounded as follows:
	\begin{align}
	\mathbb{E} & \left[\langle \nabla f(x^*;\xi_k),x^k-x^*\rangle \right] \nonumber\\
	& = \mathbb{E} \left[ \; \left\langle \mathbb{E} \left[\nabla f(x^*;\xi_k) \;|\; \Xi_k \right],x^k-x^* \right\rangle  \right] \nonumber\\
	& = \mathbb{E} \left[ \; \left\langle \nabla f(x^*), x^k - x^* \right\rangle  \right] \nonumber\\
	& = \mathbb{E} \left[ \; \langle \nabla f(x^*), \pi_X (x^k) - x^*  \rangle 	+ \langle \nabla f(x^*), x^k - \pi_X (x^k) \rangle  \right] \nonumber\\
	& \overset{O.C.}{\ge} \mathbb{E} \left[ \;  \langle \nabla f(x^*), x^k - \pi_X (x^k) \; \rangle  \right] \nonumber\\
	& \overset{C.S.}{\ge} \mathbb{E} \left[ \;  -\norm{f(x^*)} \text{dist}_X(x^k)  \right] \overset{}{\ge} -\norm{f(x^*)} \sqrt{\mathbb{E} \left[   \text{dist}_X^2(x^k)\right]}, \label{small_optcond_rhs}
	\end{align}
	where in the second inequality we used Cauchy-Schwartz inequality and in the last one we used $\mathbb{E}[Y] \le \sqrt{\mathbb{E}[Y^2]}$.
After taking full expectation in \eqref{rel_for_(ii)} and using \eqref{elemineq_lowbound}, \eqref{small_optcond_rhs} and the linear regularity property $h_{\mu}(x) := \frac{1}{2\mu} \mathbb{E}\left[ \text{dist}_{X_{\xi}}^2(x) \right] \ge \frac{\kappa}{2\mu} \text{dist}_X^2(x)$ yields the following:
	\begin{align*}
	\mathbb{E}&\left[F(x^{k+1};\xi_k)  - F^* + \frac{1}{4\mu_k}\norm{x^{k+1} - x^k}^2 \right] \\
	& \overset{\eqref{elemineq_lowbound}+\eqref{small_optcond_rhs}}{\ge} - 2\mu_k \mathbb{E}\left[  \norm{\nabla f(x^*;\xi)}^2 \right] -\norm{f(x^*)} \sqrt{\mathbb{E} \left[   \text{dist}_X^2(x^k)\right]} + \mathbb{E}[h_{4\mu_k}(x^k)] \nonumber  \\
	& \overset{lin. reg.}{\ge} - 2\mu_k\mathbb{E} [ \norm{\nabla f(x^*;\xi)}^2  -\norm{f(x^*)} \sqrt{\mathbb{E} \left[   \text{dist}_X^2(x^k)\right]} +  \frac{\kappa}{8\mu_k}\mathbb{E}[\text{dist}_X^2(x^k)] \nonumber  \\
	& \overset{\eqref{elemineq_lowbound}}{\ge} - 2\mu_k\mathbb{E}\left[\norm{\nabla f(x^*;\xi)}^2 \right] -  \frac{4\mu_k}{\kappa}\norm{\nabla f(x^*)}^2 + \frac{\kappa}{16\mu_k}\mathbb{E}[\text{dist}_X^2(x^k)].   \nonumber  
	\end{align*}
	In the last inequality we also used \eqref{elemineq_lowbound}. 
\end{proof}

\noindent Next we present the main recurrences which will finally generate our nonasymptotic convergence rates.
\begin{theorem}\label{th_reccurence}
Let Assumptions \ref{assump_basic} hold, then the SSPG sequence $\{x^k\}_{k \ge 0}$ satisfies:	
\begin{itemize}
	\item[$(i)$] Let $\mu_k \le \frac{1}{2L_f}$ for all $k \ge 0$, then:
\begin{align*}
\mathbb{E}[\norm{x^{k+1} - x^*}^2] \le  & (1 - \sigma_f \mu_k) \mathbb{E}[\norm{x^k-x^*}^2]  +  \mu_k^2 \Sigma,
\end{align*}
where $\Sigma = 2\mathbb{E} \left[  \norm{g_F(x^*;\xi)}^2\right]$.
	\item[$(ii)$] In particular, let $h(\cdot;\xi) = \mathbb{I}_{X_{\xi}}(\cdot)$ such that $\{X_{\xi}\}_{\xi \in \Omega}$ are linearly regular sets with constant $\kappa$. Then the following recurrence holds:
	\begin{align*}
	\mathbb{E}[\norm{x^{k+1} - x^*}^2] \le  & (1 - \sigma_f \mu_k) \mathbb{E}[\norm{x^k-x^*}^2]  +  \mu_k^2 \Sigma - \frac{\kappa}{8} \mathbb{E}[\text{dist}_X^2(x^k)],
	\end{align*}
	where $\Sigma = 4\mathbb{E}\left[\norm{\nabla f(x^*;\xi)}^2 \right] +  \frac{8}{\kappa}\norm{\nabla f(x^*)}^2$.
\end{itemize}
\end{theorem}
\begin{proof}
The proof result straightforwardly from Lemmas \ref{lemma_recurence} and \ref{lemma_auxres}.
\end{proof}	
\begin{remark}
Consider deterministic setting $F(\cdot;\xi) = F(\cdot)$ and $\mu_k = \frac{1}{2L_f}$, then SSPG becomes the proximal gradient algorithm and Theorem \ref{th_reccurence}$(i)$ holds with $ g_F(x^*;\xi) = g_F(x^*) = 0$, implying that $\Sigma = 0$. Thus the well-known iteration complexity estimate  $\mathcal{O}\left(\frac{L_f}{\sigma_f} \log(1/\epsilon) \right)$ \cite{Nes:13,BecTeb:09} of proximal gradient algorithm is recovered up to a constant from Theorem \ref{th_reccurence}.
\end{remark}
The above recurrences generates immediately the following convergence rates.


\begin{corrolary}\label{main_corrolary}
Under Assumption \ref{assump_basic}, the following convergence rates hold: 

\noindent $(i)$ Let $\mu_k = \frac{1}{k^{\gamma}}, \gamma \in (0,1)$ then:\quad  $\mathbb{E}[\norm{ x^{k} - x^*}^2]
\le \mathcal{O}\left(\frac{1}{k^{\gamma}}\right) $
\begin{align*}
\noindent (ii) \; \text{Let} \; \mu_k = \frac{1}{k}, \text{then} : \quad 
\mathbb{E}[\norm{ x^{k} - x^*}^2] \le
\begin{cases}
\mathcal{O}\left(\frac{1}{ k}\right)     & \text{if} \; \mu_0\sigma_{f} > e-1 \\
\mathcal{O}\left(\frac{\ln{k}}{ k}\right)     & \text{if} \; \mu_0\sigma_{f}>e-1\\
\mathcal{O}\left(\frac{1}{k}\right)^{2 \ln
	(1 + \mu_0\sigma_{f})}  & \text{if} \; \mu_0\sigma_{f} < e-1.
\end{cases}
\end{align*}
\noindent $(iii)$ For constant stepsize $\mu_k = \mu > 0$, the recurrence from Theorem \ref{th_reccurence}$(i)$ implies:
\begin{align}
\mathbb{E}[\norm{x^{k} - x^*}^2] 
&\le (1-\mu \sigma_f)^k \norm{x^0-x^*}^2 + \frac{\mu}{\sigma_f}\Sigma, \label{rate_ctstep}
\end{align}
where $\Sigma = 2\mathbb{E} \left[  \norm{g_F(x^*;\xi)}^2\right]$.

\noindent $(iv)$ Moreover, consider convex feasibility problem where $f = 0$, $h(\cdot;\xi) = \mathbb{I}_{X_{\xi}}(\cdot)$ with  $\kappa-$linearly regular sets $\{X_{\xi}\}_{\xi \in \Omega}$ and constant stepsize $\mu_k = \mu$. Then the SSPG sequence $\{x^k\}_{k \ge 0}$ converges linearly as follows:
\begin{align*}
\mathbb{E}[\text{dist}_X^2(x^{k})] 
&\le \left(1 - \frac{\kappa}{8} \right)^k \text{dist}_X^2(x^0). \end{align*}
\end{corrolary}
\begin{proof}
The proof for the first two results $(i)$ and $(ii)$ follows similar lines with \cite[Corrolary 15]{PatNec:18}. However, for completeness we present it in appendix. 	

\vspace{5pt}

\noindent For $(iii)$, notice that Theorem \ref{th_reccurence}$(i)$ straightforwardly implies:
	\begin{align}
	\mathbb{E}[\norm{x^{k} - x^*}^2] & \le (1 - \sigma_f \mu) \mathbb{E}[\norm{x^{k-1}-x^*}^2]  +  \mu^2 \Sigma \nonumber \\
	& \le (1-\mu \sigma_f)^k \norm{x^0-x^*}^2 + \mu^2 \Sigma\sum\limits_{i=0}^{k-1} (1-\mu\sigma_f)^{i} \nonumber\\
	&\le (1-\mu \sigma_f)^k \norm{x^0-x^*}^2 + \frac{\mu}{\sigma_f}\Sigma [1-(1-\mu\sigma_f)^k] \nonumber \\ 
	&\le (1-\mu \sigma_f)^k \norm{x^0-x^*}^2 + \frac{\mu}{\sigma_f}\Sigma. \nonumber 
	\end{align}
Now, let CFP settings $f=0, h(\cdot;\xi) = \mathbb{I}_{X_{\xi}}(\cdot)$ hold. Theorem \ref{th_reccurence}$(ii)$ states that:
\begin{align*}
\mathbb{E}[\norm{x^{k+1} - x^*}^2] \le  & \mathbb{E}[\norm{x^k-x^*}^2] - \frac{\kappa}{8} \mathbb{E}[\text{dist}_X^2(x^k)] \quad \forall x^* \in X^*.
\end{align*}	
Since in this case, $X = \bigcap_{\xi} X_{\xi} = X^*$, then by choosing  $x^* = \pi_X(x^k)$ and using in the LHS that $\norm{x^{k+1}- \pi_X(x^k)} \ge \text{dist}_X(x^{k+1})$, then we obtain:
\begin{align*}
\mathbb{E}[\text{dist}_X^2(x^{k+1})] \le \mathbb{E}[\norm{x^{k+1} - \pi_X(x^k)}^2] \le & \left(1 - \frac{\kappa}{8} \right) \mathbb{E}[\text{dist}_X^2(x^k)],
\end{align*}	
which yields the linear convergence rate of SSPG.
\end{proof}
\begin{remark}	
\noindent \noindent Although sublinear $\mathcal{O}(1/k)$ convergence rates for strongly convex objectives are typically obtained in literature for many first-order stochastic schemes \cite{MouBac:11,PatNec:18}, these rates are novel  due to their generalization potential.

\noindent Regarding second rate $(ii)$, although it expresses a geometric decrease of the initial residual term, this rate states that, after $\mathcal{O}\left(\frac{1}{\mu \sigma_f} \log \left(\frac{1}{\epsilon}\right) \right)$ iterations, the sequence $\{x^k\}_{k \ge 0}$ will remain (in expectation) in a bounded neighborhood of the optimal point $\{x: \norm{x - x^*}^2 \le \frac{\mu}{\sigma_f}\Sigma\}$. This fact suggests that only sufficiently small constant stepsizes guarantee the convergence of SSPG sequence.

\noindent In the convex feasibility setting, SSPG reduces to Randomized Alternating Projections algorithm for which the obtained linear rate obeys the rates from the literature up to a constant. However, we believe that using some refinements of proof arguments there might be obtained optimal rates w.r.t. the constants.


\section{Application to Parametric Sparse Representation}

Sparse representation~\cite{Ela:10}
starts with the given signal $y \in \rset^m$
and aims to find the sparse signal $x\in\rset^n$
by projecting $y$ to a much smaller subspace
through the overcomplete dictionary $T \in \rset^{m \times n}$. Among the many Dictionary Learning techniques, we focus on the multi-parametric cosparse model proposed in \cite{StoIro:19}, since it could efficiently illustrates the empirical capabilities of SSPG.
The multi-parametric cosparse  representation problem is given by:
\be
\begin{aligned}
	& \underset{x}{\min}
	& & \|Tx-y\|^2_2 \\
	& \text{s.t.}
	& & \|\Delta x\|_1 \le \delta,
\end{aligned}
\label{eq:optprob}
\ee
where $T$ and $x$ correspond to the dictionary and, respectively, the resulting sparse representation,
with sparsity being imposed on a scaled subspace $\Delta x$ with $\Delta\in\rset^{p\times n}$.
In pursuit of \eref{problem_intro},
we move to the exact penalty problem $\min_x \frac{1}{2m}\norm{Tx - y}^2_2 + \lambda \norm{\Delta x }_1$. In order to limit the solution norm we further regularize the unconstrained objective using an $\ell_2$ term as follows:
\be
\begin{aligned}
	& \underset{x}{\min} & & \frac{1}{2m}\|Tx-y\|^2_2 +
	\lambda\|\Delta x\|_1 +
	\frac{\alpha}{2} \|x\|^2_2.
\end{aligned}
\label{eq:strongprob}
\ee
The decomposition which puts the above formulation into model \eqref{problem_intro} consists of:
\be
f(x;\xi) = \frac{1}{2} (T_\xi x-y_\xi)^2_2 + 
\frac{\alpha}{2} \|x\|^2_2
\label{eq:sp_smooth}
\ee
where $T_\xi$ represents line $\xi$ of matrix $T$, and 
\be
h(x; \xi) = m\lambda |\Delta_\xi x|.
\label{eq:sp_nonsmooth}
\ee
To compute the SSPG iteration for the sparse representation problem, we note that
the proximal operator $\text{prox}_{h,\mu}(x;\xi) =   \arg\min\limits_{z \in \rset^n} m \lambda |\Delta_\xi z| + \frac{1}{2\mu} \norm{z - x}^2$ is given by
\begin{align*}
	\text{prox}_{h,\mu}(x;\xi) &   = 
	\begin{cases}
		x - \frac{\Delta_{\xi}x}{\norm{\Delta_{\xi}}^2 } \Delta_\xi^T \; & \text{if} \; \frac{|\Delta_\xi x|}{m\lambda \norm{\Delta_\xi}^2} \le \mu \\
		x - \mu m \lambda \text{sgn}(\Delta_{\xi}x) \Delta_\xi^T \; & \text{if} \; \frac{|\Delta_\xi x|}{m\lambda \norm{\Delta_\xi}^2} > \mu \\
	\end{cases}
\end{align*}
Also, the gradient of the smooth function $f(\cdot;\xi)$ is given by
\begin{align*}
	\nabla f(x; \xi) =  \left(T_{\xi}^T T_\xi + \alpha I_n \right)x - T_{\xi}^T y_\xi
\end{align*}
and with that we are ready to formulate the resulting particular variant of SSPG.
\begin{flushleft}
	\textbf{SSPG - Sparse Representation  (SSPG-SR)}: \quad For $k\geq 0$ compute \\
	1. Choose randomly $\xi_k \in \Omega$ w.r.t. probability distribution $\mathbb{P}$\\
	2. Update: 
	\begin{align*}
		y^{k}  & =  \left[ I_n - \mu_k \left(T_{\xi_k}^T T_{\xi_k} + \alpha I_n \right)\right] x^k + \mu_k T_{\xi_k}^T y_{\xi_k}, \quad \beta_k = \frac{\Delta_{\xi_k}y^k}{\norm{\Delta_{\xi_k}}^2} \\
		x^{k+1}  & =  \begin{cases}
			y^k - \beta_k  \Delta_{\xi_k}^T \; & \text{if} \; |\beta_k| \le m\lambda \mu \\
			y^k - \mu m \lambda \text{sgn}(\beta_k) \Delta_{\xi_k}^T \; & \text{if} \; |\beta_k| >  m\lambda \mu\\
		\end{cases}	
	\end{align*}
	3. If the stopping criterion holds, then \textbf{STOP}, otherwise $k = k+1$.
\end{flushleft}

\noindent We use randomly generated data with a standard normal distribution. \footnotemark \footnotetext{Data generating code available at https://github.com/pirofti/SSPG  }	 In the following we present numerical simulations that show-case the application of SSPG to the multi-parametric sparse representation problem and compare it to CVX and Proximal Gradient method (denoted { \verb prox }) \cite{Nes:13}. Notice that at each iteration Proximal Gradient method requires, given current point $x^k_{PG}$, the computation of: 
\begin{itemize}
\item gradient $\nabla f(x^k_{PG}) = \frac{1}{m}T^T(Tx^k_{PG}-y)$ 
\item proximal update $z(x^k_{PG}) = \arg\min_z \lambda\norm{\Delta z}_1 + \frac{L}{2}\norm{z-(x^k_{PG} - \frac{1}{L}\nabla f(x^k_{PG}))}^2$
\end{itemize} 

\pdffig[\textwidth]{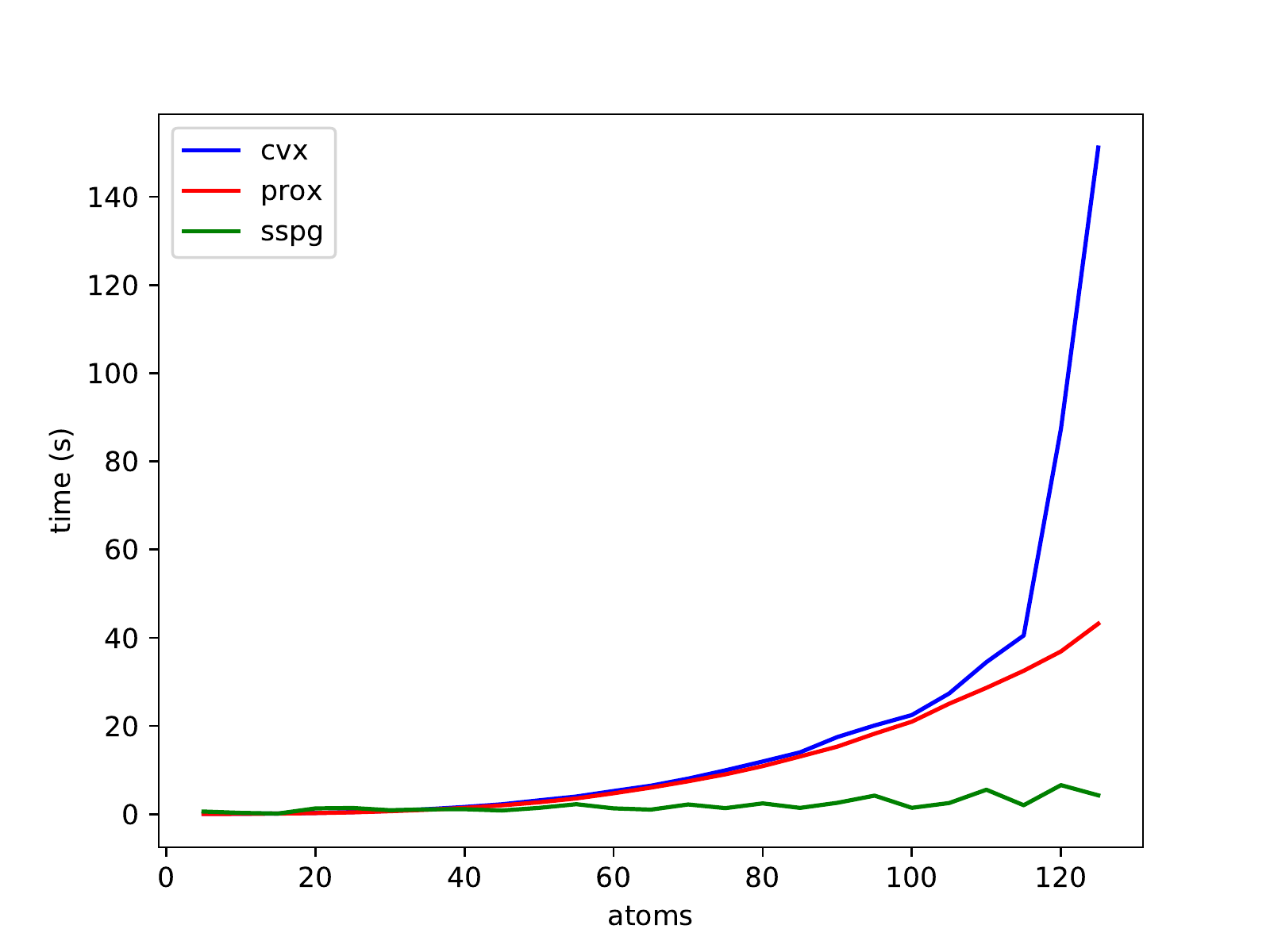}
{Increasing problem dimension.
	Features are 4 times the number of atoms.
	($\alpha=0.5$, $\lambda=5$, $\varepsilon=10^{-6}$)}
{fig:sspg_atoms}

Our first experiment investigates the impact of the problem size on execution time.
We vary the atoms in the dictionary starting from $n=5$ up to $n=125$ in increments of 5
and maintain a ratio of 4 features per atom such that $m=4n$.
First CVX is used to determine $x^\star$ within a $\varepsilon=10^{-6}$ margin.
Then Proximal Gradient method and SSPG-SR are executed until $x^\star$ is reached with the same approximation error.
Figure \ref{fig:sspg_atoms} depicts the results.
The experiment is stopped after $n=125$ when CVX execution times become too large.
We can clearly see a similar yet gentler increasing trend for Proximal Gradient method (prox) and indeed other experiments have shown it to reach an impas shortly after.
SSPG-SR on the other hand maintains a steady pace throughout the tests.

\pdffig[\textwidth]{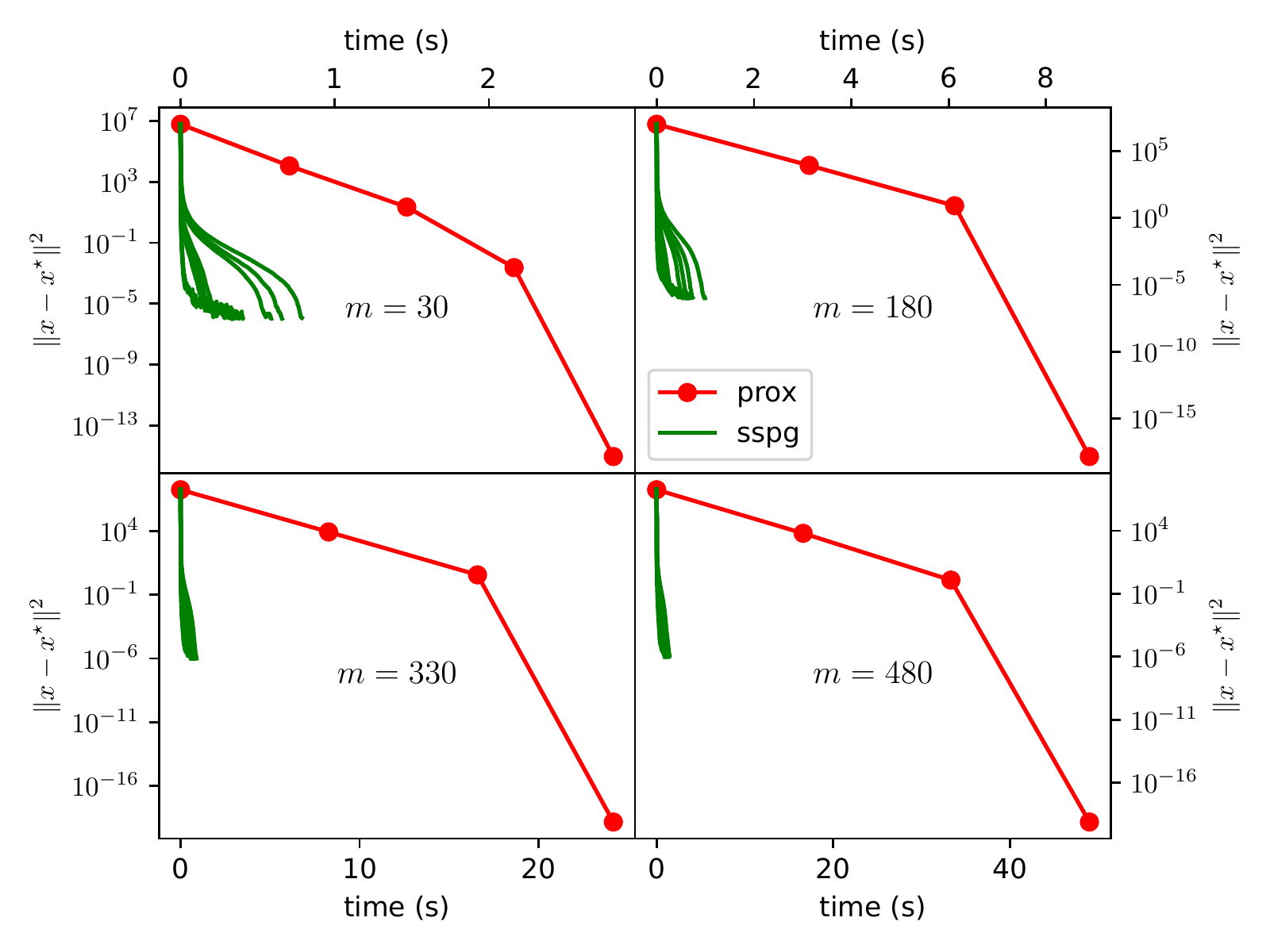}
{Proximal Gradient method versus 10 rounds of SSPG iterations.
	Features are 6 times the number of atoms.
	($\alpha=0.2$, $\lambda=5\cdot10^{-4}$, $\varepsilon=10^{-6}$)}
{fig:sspg_iters}

Next, we focus on the iterations of Proximal Gradient method and SSGP-SR.
To this end we design a similar experiment consisting of 4 tests
where we vary the problem size by increasing the number of features
with a ratio of 6 features per atom ($m=6n$).
As before,
we use CVX to determine $x^\star$ with $\varepsilon=10^{-6}$.
With identical parameters and initialization
we execute 10 rounds of SSPG-SR and compare them to the Proximal Gradient method iterations.
We time the execution of each iteration and record its progress $\norm{x^k-x^\star}$.
The result is shown in Figure \ref{fig:sspg_iters}.
In all 4 panels it is clearly seen that
the 10 SSPG-SR rounds perform much faster than Proximal Gradient method at least in a first phase.
Further increases of the problem dimension lead to a stall in Proximal Gradient method.
While SSPG-SR reaches a solution in less than a second in all the tests,
Proximal Gradient method goes from 5 seconds in the first, to 9, 24 and 50 in the last.
We note that even though the stopping criterion was the same for both methods ($\varepsilon=10^{-6})$,
Proximal Gradient method gets us much closer to $x^\star$ because, in our parameter setting, the SR problem is well-conditioned.

\section{Conclusion}

\noindent In this paper we presented preliminary guarantees for stochastic gradient schemes with stochastic proximal updates, which unify some well-known schemes in the literature. For future work, would be interesting to analyze convergence rate on (non)convex functions satisfying more relaxed convexity conditions (i.e. quadratic growth) and the empirical behaviour of SSPG scheme under various stepsize choices.

\end{remark}

\bibliographystyle{plain}
\bibliography{sppg}

\section{Appendix}

\begin{proof}[of Corollary \ref{main_corrolary}]
	For simplicity denote $\theta_k = (1 - \mu_k\sigma_{f})$ , then Theorem \ref{th_reccurence} implies that:
	\begin{align*}
	\mathbb{E}\left[\norm{x^{k+1}-x^*}^2 \right] 
	& \le  \left(\prod_{i=0}^k \theta_i\right) \norm{x^0-x^*}^2 + \Sigma \sum\limits_{i=0}^k \left(\prod\limits_{j=i+1}^{k} \theta_j\right) \mu_i^2.
	\end{align*}
	By using the Bernoulli inequality $ 1- tx \le \frac{1}{1 + tx} \le (1 + x)^{-t}$ for $t \in [0,1], x \ge 0$, then we have:
	\begin{align}\label{bern_conseq1}
	\prod\limits_{i=l}^u \theta_i 
	& = \prod\limits_{i=l}^u \left(1 - \frac{\mu_0}{i^{\gamma}} \sigma_{f}\right)  \le \prod\limits_{i=l}^u (1 + \mu_0 \sigma_f)^{-1/i^{\gamma}} 
	=  (1 + \mu_0 \sigma_{f})^{- \sum\limits_{i=l}^u \frac{1}{i^{\gamma}}}.
	\end{align}
	On the other hand, if we use the lower bound
	\begin{align}\label{bern_conseq2}
	\sum\limits_{i=l}^u \frac{1}{i^{\gamma}} \ge \int\limits_{l}^{u + 1} \frac{1}{\tau^{\gamma}} d\tau =  \varphi_{1-\gamma}(u+1) - \varphi_{1-\gamma}(l).
	\end{align}
	then we can finally derive:
	\begin{align*}
	& \sum\limits_{i=0}^k \left(\prod\limits_{j=i+1}^{k} \theta_j\right) \mu_i^2  = \sum\limits_{i=0}^m \left(\prod\limits_{j=i+1}^{k} \theta_j\right) \mu_i^2 + \sum\limits_{i=m+1}^k \left(\prod\limits_{j=i+1}^{k} \theta_j\right) \mu_i^2\\
	& \overset{\eqref{bern_conseq1} + \eqref{bern_conseq2}}{\le} \sum\limits_{i=0}^m (1 + \mu_0 \sigma_f)^{  \varphi_{1-\gamma}(i+1) - \varphi_{1-\gamma}(k)  } \mu_i^2 + \mu_{m+1} \sum\limits_{i=m+1}^k \left[\prod\limits_{j=i+1}^{k} (1 - \mu_j\sigma_f) \right] \mu_i \\
	& \le  (1 + \mu_0 \sigma_f)^{  \varphi_{1-\gamma}(m) - \varphi_{1-\gamma}(k)  } \sum\limits_{i=0}^m  \mu_i^2\\ & \hspace{3cm} + \frac{\mu_{m+1}}{\sigma_f} \sum\limits_{i=m+1}^k \left[\prod\limits_{j=i+1}^{k} (1 - \mu_j\sigma_f) \right] (1 - (1- \sigma_f\mu_i)) \\
	& = (1 + \mu_0 \sigma_f)^{  \varphi_{1-\gamma}(m) - \varphi_{1-\gamma}(k)  } \mu_0^2 \sum\limits_{i=0}^m  \frac{1}{i^{2\gamma}} \\    & \hspace{3cm} + \frac{\mu_{m+1}}{\sigma_f} \sum\limits_{i=m+1}^k \left[\prod\limits_{j=i+1}^{k} (1 - \mu_j\sigma_f)  - \prod\limits_{j=i}^{k} (1 - \mu_j\sigma_f) \right] \\
	& \le (1 + \mu_0 \sigma_f)^{  \varphi_{1-\gamma}(m) - \varphi_{1-\gamma}(k)  } \frac{m^{1- 2\gamma} - 1}{1 - 2\gamma} + \frac{\mu_{m+1}}{\sigma_f} \left[1  - \prod\limits_{j=m+1}^{k} (1 - \mu_j\sigma_f) \right] \\
	& \le (1 + \mu_0 \sigma_f)^{  \varphi_{1-\gamma}(m) - \varphi_{1-\gamma}(k)  } \varphi_{1 - 2\gamma}(m) + \frac{\mu_{m+1}}{\sigma_f}.
	\end{align*}
	By denoting the second constant $\tilde{\theta}_0 = \frac{1}{1+\mu_0 \sigma_f}$, then the last relation implies the following bound:
	\begin{align*}
	\mathbb{E}\left[\norm{x^{k+1}-x^*}^2\right] 
	\le \tilde{\theta}_0^{\varphi_{1-\gamma}(k)} \norm{x^{0}-x^*}^2 + \tilde{\theta}_0^{  \varphi_{1-\gamma}(k) - \varphi_{1-\gamma}(m)  } \varphi_{1 - 2\gamma}(m)\Sigma + \frac{\mu_{m+1}}{\sigma_f} \Sigma.
	\end{align*}
	Denote $r_k^2 = \mathbb{E}[\norm{x^k-x^*}^2]$.
	To derive an explicit convergence rate order we analyze upper bounds on function $\phi$. 
	
	\noindent $(i) \; $  First assume that $\gamma \in (0, \frac{1}{2})$. This implies that $1 - 2\gamma > 0$ and that:
	\begin{align}\label{corr_main_sc_est1}
	\varphi_{1-2\gamma}\left(\left\lfloor \frac{k}{2} \right \rfloor \right) \le \varphi_{1-2\gamma}\left(\frac{k}{2}\right) = \frac{\left(\frac{k}{2} \right)^{1-2\gamma} - 1}{1-2\gamma}\le \frac{\left(\frac{k}{2}  \right)^{1-2\gamma}}{1-2\gamma}.
	\end{align}
	On the other hand, by using the inequality $e^{-x} \le \frac{1}{1 + x}$ for all $x \ge 0$, we obtain:
	\begin{align*}
	& \tilde{\theta}_0^{\varphi_{1-\gamma}(k) - \varphi_{1-\gamma}(\frac{k-2}{2})}
	\varphi_{1-2\gamma}\left(\frac{k}{2}\right)
	=  e^{(\varphi_{1-\gamma}(k) - \varphi_{1-\gamma}(\frac{k-2}{2}))\ln {\tilde{\theta}_0}} \varphi_{1-2\gamma}\left(\frac{k}{2} \right) \\
	&\le \frac{\varphi_{1-2\gamma}\left(\frac{k}{2} \right)}{1 +
		[\varphi_{1-\gamma}(k) -
		\varphi_{1-\gamma}(\frac{k}{2}-1)]\ln{\frac{1}{\tilde{\theta}_0}}}
	\overset{\eqref{corr_main_sc_est1}}{\le} \frac{\frac{k^{1-2\gamma}}{2^{1-2\gamma} (1-2\gamma)}  }{\frac{1}{1-\gamma}[k^{1-\gamma} - (\frac{k}{2}-1)^{1-\gamma}]\ln{\frac{1}{\tilde{\theta}_0}}} \\
	& = \frac{\frac{k^{1-2\gamma}}{2^{1-2\gamma }
			(1-2\gamma)}}{\frac{k^{1-\gamma}}{1-\gamma}[1
		- (\frac{1}{6})^{1-\gamma}]\ln{\frac{1}{\tilde{\theta}_0}}}
	= \frac{1-\gamma}{1-2\gamma}\frac{2^{\gamma}k^{-\gamma}}{2^{1-2\gamma}[1 - (\frac{1}{6})^{1-\gamma}]\ln{\frac{1}{\theta_0}}} = \mathcal{O}\left( \frac{1}{k^{\gamma}}\right).
	\end{align*}
	Therefore, in this case, the overall rate will be given by:
	$$r_{k+1}^2 \le \theta_0^{\mathcal{O}(k^{1-\gamma})}r_0^2 + \mathcal{O}\left( \frac{1}{k^{\gamma}}\right) \approx\mathcal{O}\left( \frac{1}{k^{\gamma}}\right) .$$
	If $\gamma = \frac{1}{2}$, then the definition of
	$\varphi_{1-2\gamma}(\frac{k}{2})$  provides that:
	$$r_{k+1}^2 \le \tilde{\theta}_0^{\mathcal{O}(\sqrt{k})}r_0^2 + \tilde{\theta}_0^{\mathcal{O}(\sqrt{k})}\mathcal{O}(\ln{k}) +  \mathcal{O}\left( \frac{1}{\sqrt{k}}\right) \approx\mathcal{O}\left( \frac{1}{\sqrt{k}}\right) .$$
	When $\gamma \in (\frac{1}{2}, 1)$, it is obvious that
	$\varphi_{1-2\gamma}\left(\frac{k}{2}\right) \le \frac{1}{2\gamma
		- 1}$ and therefore the order of the convergence rate  changes into:
	$$r_{k+1}^2 \le \tilde{\theta}_0^{\mathcal{O}(k^{1-\gamma})}[r_0^2 + \mathcal{O}(1)] + \mathcal{O}\left( \frac{1}{k^{\gamma}}\right) \approx\mathcal{O}\left( \frac{1}{k^{\gamma}}\right).$$
	
	\noindent $(ii)\;$ Lastly, if $\gamma = 1$, by using $\tilde{\theta}_0^{\ln{k+1}} \le
	\left(\frac{1}{k}\right)^{\ln{\frac{1}{\tilde{\theta}_0}}}$ we obtain the
	second part of our result.
\end{proof}

\end{document}